\documentclass{amsart}

\usepackage{ae,aecompl}
\usepackage{esint}
\usepackage{graphicx}
\usepackage[all,cmtip]{xy}
\usepackage{amsmath, amscd}
\usepackage{booktabs}
\usepackage{amsthm}
\usepackage{amssymb}
\usepackage{amsfonts}
\usepackage{qsymbols}
\usepackage{latexsym}
\usepackage{mathrsfs}
\usepackage{cite}
\usepackage{color}
\usepackage{url}
\usepackage{enumerate}
\usepackage{undertilde}

\usepackage{verbatim}
\usepackage[draft=false, colorlinks=true]{hyperref}

\allowdisplaybreaks

\newcommand {\bmo}{\mathrm{bmo}}

\newcommand {\ud}{\mathrm{d}}

\newcommand {\veps}{\varepsilon}

\newcommand {\F}{\mathcal{F}}

\newcommand {\HT}{\mathcal{H}}
\newcommand {\Hp}{\mathcal{H}^{p}_{FIO}(\Rn)}
   
\newcommand {\Hps}{\mathcal{H}^{s,p}_{FIO}(\Rn)}

\newcommand {\rb}{\rangle}
\newcommand {\lb}{{\langle}}
\newcommand {\La}{{\mathcal{L}}}
\newcommand {\loc}{{\mathrm{loc}}}

\newcommand {\N}{{{\mathbb N}}}

\newcommand {\ph}{\varphi}

\newcommand {\R}{\mathbb R}

\newcommand {\Rn}{\mathbb{R}^{n}}

\newcommand {\rank}{\mathrm{rank}}

\newcommand {\supp}{\mathrm{supp}}

\newcommand {\Sp}{S^{*}\Rn}

\newcommand {\Sw}{\mathcal{S}}

\newcommand {\Tp}{T^{*}\Rn}

\newcommand {\w}{\omega}

\newcommand {\Z}{{\mathbb Z}}

\newcommand {\vanish}[1]{\relax}

\newcommand{\wh}{\widehat}

\DeclareFontFamily{U}{mathx}{\hyphenchar\font45}
\DeclareFontShape{U}{mathx}{m}{n}{
      <5> <6> <7> <8> <9> <10>
      <10.95> <12> <14.4> <17.28> <20.74> <24.88>
      mathx10
      }{}
\DeclareSymbolFont{mathx}{U}{mathx}{m}{n}
\DeclareFontSubstitution{U}{mathx}{m}{n}
\DeclareMathAccent{\widecheck}{0}{mathx}{"71}

\newtheorem{theorem}{Theorem}[section]
\newtheorem{lemma}[theorem]{Lemma}
\newtheorem{proposition}[theorem]{Proposition}
\newtheorem{corollary}[theorem]{Corollary}

\theoremstyle{definition}
\newtheorem{definition}[theorem]{Definition}
\newtheorem{remark}[theorem]{Remark}

\numberwithin{equation}{section}
\protected\def\ignorethis#1\endignorethis{}
\let\endignorethis\relax

\title[Local smoothing and Hardy spaces for FIOs]{Local smoothing and Hardy spaces for Fourier integral operators}

\author{Jan Rozendaal}
\address{Institute of Mathematics, Polish Academy of Sciences\\
ul.~\'{S}niadeckich 8\\
00-656 Warsaw\\
Poland}
\email{jrozendaal@impan.pl}



\thanks{This research was supported by grant DP160100941 of the Australian Research Council, and by NCN grant UMO2017/27/B/ST1/00078. The research leading to these results has received funding from the Norwegian Financial Mechanism 2014-2021, grant 2020/37/K/ST1/02765.}

\begin{document}

\begin{abstract}
We show that the Hardy spaces for Fourier integral operators form natural spaces of initial data when applying $\ell^{p}$-decoupling inequalities to local smoothing for the wave equation. This yields new local smoothing estimates which, in a quantified manner, improve the bounds in the local smoothing conjecture on $\Rn$ for $p\geq 2(n+1)/(n-1)$, and complement them for $2<p<2(n+1)/(n-1)$. These estimates are invariant under application of Fourier integral operators, and they are essentially sharp.
\end{abstract}

\maketitle

\section{Introduction}

This article concerns the local smoothing conjecture for the Euclidean wave equation on $\Rn$. The phenomenon of local smoothing revolves around determining, for each $1<p<\infty$, the minimal $s\in\R$ for which there exists a $C\geq0$ such that
\begin{equation}\label{eq:conj}
\Big(\int_{0}^{1}\|e^{it\sqrt{-\Delta}}f\|_{L^{p}(\Rn)}^{p}\ud t\Big)^{1/p}\leq C\|f\|_{W^{s,p}(\Rn)}
\end{equation}
for all $f\in W^{s,p}(\Rn)$. The optimal fixed-time estimates for the half-wave propagator $e^{it\sqrt{-\Delta}}$ from \cite{Peral80,Miyachi80a} imply that one may choose $s=2s(p)$, where we write
\[
s(p):=\frac{n-1}{2}\Big|\frac{1}{2}-\frac{1}{p}\Big|
\]
throughout. This estimate is sharp for $1<p\leq 2$, but it can be improved for $p>2$ and $n\geq2$. In fact, the local smoothing conjecture, as originally formulated by Sogge in~\cite{Sogge91}, stipulates that \eqref{eq:conj} should hold with $s=\sigma(p)+\veps$ for each $\veps>0$, where $\sigma(p)=0$ for $2<p\leq 2n/(n-1)$, and $\sigma(p)=2s(p)-1/p$ for $p>2n/(n-1)$. 

Although the local smoothing conjecture is still open in full generality, there are many partial results available, and the conjecture was proved by Guth, Wang and Zhang \cite{GuWaZh20} for $n=2$. For the purposes of the present article, however, it is relevant to highlight the article~\cite{Bourgain-Demeter15} by Bourgain and Demeter. Building on work by Wolff~\cite{Wolff00}, they showed that \eqref{eq:conj} holds with $s=\sigma(p)+\veps$ for $p\geq 2(n+1)/(n-1)$, by proving the $\ell^{p}$-decoupling inequality. More precisely, one has
\begin{equation}\label{eq:conj2}
\Big(\int_{0}^{1}\|e^{it\sqrt{-\Delta}}f\|_{L^{p}(\Rn)}^{p}\ud t\Big)^{1/p}\leq C\|f\|_{W^{d(p)+\veps,p}(\Rn)},
\end{equation}
where
\[
d(p):=\begin{cases}
2s(p)-\frac{1}{p}&\text{if }p\geq \frac{2(n+1)}{n-1},\\
s(p)&\text{if }2\leq p<\frac{2(n+1)}{n-1}.
\end{cases}
\]

\subsection*{Main results}

The local smoothing conjecture is sharp, in the sense that \eqref{eq:conj} does not hold for $s<\sigma(p)$ (see e.g.~\cite[Remark 2.6]{BeHiSo21}). Nonetheless, in this article we obtain improved local smoothing estimates for $p\geq2(n+1)/(n-1)$, by working with a different space of initial data in \eqref{eq:conj2}. Our main result, formulated in terms of the Sobolev spaces $\HT^{s,p}_{FIO}(\Rn)=\lb D\rb^{-s}\HT^{p}_{FIO}(\Rn)$ over the Hardy spaces for Fourier integral operators $\Hp$ (see Definition \ref{def:HpFIO}), is as follows.

\begin{theorem}\label{thm:localsmooth}
Let $p\in(2,\infty)$ and $\veps>0$. Then there exists a $C>0$ such that 
\begin{equation}\label{eq:localsmooth}
\Big(\int_{0}^{1}\|e^{it\sqrt{-\Delta}}f\|_{L^{p}(\Rn)}^{p}\ud t\Big)^{1/p}\leq C\|f\|_{\HT^{d(p)-s(p)+\veps,p}_{FIO}(\Rn)}
\end{equation}
for all $f\in \HT^{d(p)-s(p)+\veps,p}_{FIO}(\Rn)$.
\end{theorem}

Theorem \ref{thm:localsmooth} is a special case of Theorem \ref{thm:localsmoothmain}, which considers the more general class of operators of the form $e^{it\phi(D)}$, where $\phi\in C^{\infty}(\Rn\setminus\{0\})$ is positively homogeneous of degree $1$ and such that $\rank(\partial_{\xi\xi}^{2}\phi)\equiv n-1$. Moreover, the exponent $d(p)-s(p)$ in \eqref{eq:localsmooth} is sharp, as is shown in Theorem \ref{thm:sharp}.

The Hardy space $\HT^{1}_{FIO}(\Rn)$ for Fourier integral operators (FIOs) was introduced by Smith in \cite{Smith98a}, and decades later his construction was extended to a scale $(\HT^{p}_{FIO}(\Rn))_{1\leq p\leq \infty}$ of invariant spaces for FIOs by Hassell, Portal and the author in \cite{HaPoRo20}. More precisely, $\HT^{p}_{FIO}(\Rn)$ is invariant under FIOs of order zero which have a compactly supported Schwartz kernel and are associated with a local canonical graph, and one has
\begin{equation}\label{eq:Sobolev}
W^{s(p),p}(\Rn)\subseteq \Hp\subseteq W^{-s(p),p}(\Rn)
\end{equation}
for all $1<p<\infty$ and $s\in\R$, with the natural modifications involving the local Hardy space $\HT^{1}(\Rn)$ for $p=1$, and $\bmo(\Rn)$ for $p=\infty$. In particular, 
\[
W^{d(p)+\veps,p}(\Rn)\subseteq \HT^{d(p)-s(p)+\veps,p}_{FIO}(\Rn)\subseteq W^{d(p)-2s(p)+\veps,p}(\Rn)
\]
for $2<p<\infty$, and \eqref{eq:localsmooth} recovers \eqref{eq:conj2}. However, since the exponents in \eqref{eq:Sobolev} are sharp, \eqref{eq:localsmooth} is in fact a strict improvement of \eqref{eq:conj2}. And since $d(p)=\sigma(p)$ for $p\geq 2(n+1)/(n-1)$, Theorem \ref{thm:localsmooth} improves upon the local smoothing conjecture for such $p$. 

Note that $d(p)-s(p)=\max(0,s(p)-1/p)$, 
\[
d(p)-s(p)=0=\sigma(p) \quad\text{for}\quad2<p\leq 2\frac{n}{n-1},
\]
and 
\[
d(p)-s(p)=0<2s(p)-\frac{1}{p}=\sigma(p)<d(p)\quad \text{for}\quad 2\frac{n}{n-1}<p<2\frac{n+1}{n-1}.
\]
Hence the sharpness of the embeddings in \eqref{eq:Sobolev} implies that, for $2<p<2(n+1)/(n-1)$, \eqref{eq:localsmooth} neither follows from the local smoothing conjecture nor implies it. 

It was shown in \cite{HeNaSe11} that the $\veps$ loss in \eqref{eq:conj2} can be removed for $n\geq 4$ and $p>2+\frac{4}{n-3}>2\frac{n+1}{n-1}$. It is not clear whether one may also let $\veps=0$ in \eqref{eq:localsmooth} for certain $p$. For $p>2+\frac{4}{n-3}$, by the sharpness of the exponents in \eqref{eq:Sobolev}, Theorem \ref{thm:localsmooth} neither follows from the local smoothing result in \cite{HeNaSe11} nor implies it. 

\subsection*{Applications}

The invariance of $\Hp$ under FIOs and related operators shows that \eqref{eq:localsmooth} and \eqref{eq:conj2} are invariant under application of such operators to the initial data. Corollary \ref{cor:bddest} collects a few instances of this paradigm.

In particular, the Cauchy problem associated to the half-wave equation $\partial_{t}u(t)=i\sqrt{-\Delta}u(t)$ is well posed on $\Hps$. Hence in Theorem \ref{thm:localsmooth} one in fact has
\[
[t\mapsto e^{it\sqrt{-\Delta}}f]\in C(\R;\HT^{d(p)-s(p)+\veps,p}_{FIO}(\Rn))\cap L^{p}_{\loc}(\R;L^{p}(\Rn))
\]
for all $f\in \HT^{d(p)-s(p)+\veps,p}_{FIO}(\Rn)$. Here it is relevant to note that, since $(e^{it\sqrt{-\Delta}})_{t\in\R}$ is a group, one has $e^{it\sqrt{-\Delta}}f\in\HT^{r,p}_{FIO}(\Rn)$ for some $t\in\R$ if and only if this holds for all $t\in\R$, and no smoothing can occur. Instead, the half-wave group has a smoothing effect on initial data with respect to the Sobolev embeddings in \eqref{eq:Sobolev}.

Local smoothing estimates have applications to various problems in harmonic and geometric analysis, such as spherical maximal theorems (see \cite{BeHiSo21,Sogge17}). It would be of interest to examine applications of the results in this article to these areas.

In \cite{Schippa22}, Schippa obtained local smoothing estimates for the Schr\"{o}dinger equation using invariant spaces for Schr\"{o}dinger propagators, the modulation spaces. He then used those estimates to prove new local and global well-posedness results for the nonlinear Schr\"{o}dinger equation. Very recently, in \cite{Rozendaal-Schippa22}, Schippa and the author applied this approach to function spaces related to the Hardy spaces for FIOs, to prove new local and global well-posedness results for nonlinear wave equations.

\subsection*{Idea behind the proof of Theorem \ref{thm:localsmooth}}

The proof exploits the full strength of the $\ell^{p}$-decoupling inequality, a corollary of the $\ell^{2}$-decoupling inequality from \cite{Bourgain-Demeter15}. This inequality bounds, for a suitable $g\in\Sw(\R)$ and for initial data $f$ with frequency support in a dyadic annulus, the $L^{p}(\R^{n+1})$ norm of $g(t)e^{it\sqrt{-\Delta}}f(x)$ by an $\ell^{p}$ sum involving a dyadic-parabolic decomposition of the frequency support of $f$. This dyadic-parabolic decomposition goes back to \cite{Fefferman73b} and plays a key role in the proof of the optimal $L^{p}$ regularity of Fourier integral operators \cite{SeSoSt91}. On each dyadic annulus one makes a different discrete decomposition, and the right-hand side of the decoupling inequality involves a different \emph{decoupling norm} on each annulus. 

By contrast, the Hardy spaces for FIOs involve a continuous dyadic-parabolic decomposition which is the same on each frequency annulus, and one observation in this article is that the discrete and continuous decompositions are equivalent on dyadic annuli (see Proposition \ref{prop:equivalent}). It then follows from the fact that the half-wave group is polynomially bounded on the Hardy spaces for FIOs that, when applied to the function $g(t)e^{it\sqrt{-\Delta}}f(x)$ in the setting of Theorem \ref{thm:localsmooth}, the right-hand side of the $\ell^{p}$-decoupling inequality is equivalent to the $\HT^{d(p)-s(p)+\veps,p}_{FIO}(\Rn)$ norm of $f$. In fact, Theorem \ref{thm:localsmooth} is a consequence of Theorem \ref{thm:localsmoothmain}, which is proved in the same manner but relies on a more general class of $\ell^{p}$-decoupling inequalities, from \cite{Bourgain-Demeter17}. As is shown in Corollary \ref{cor:wave}, the Hardy spaces for FIOs are adapted to each of these decoupling inequalities.

Hence, at least when restricted to dyadic frequency annuli, the Hardy spaces for FIOs form the \emph{largest} space of initial data for which one can obtain local smoothing estimates when applying $\ell^{p}$-decoupling inequalities in the manner in which they are typically used.  Moreover, since the exponent $d(p)-s(p)$ in \eqref{eq:localsmooth} is sharp, for $2<p<2(n+1)/(n-1)$ one cannot improve upon the local smoothing conjecture by using the Hardy spaces for FIOs as spaces of initial data. Here it is relevant to note that the example which shows that the exponent $d(p)-s(p)$ is sharp involves, for $2<p<2(n+1)/(n-1)$, a square function which was used in \cite{GuWaZh20} to prove the full local smoothing conjecture for $n=2$ (see Lemma \ref{lem:squarefun}). On the other hand, while the local smoothing conjecture  is still open for $n\geq3$, this article contains the essentially optimal local smoothing estimates with the Hardy spaces for FIOs as spaces of initial data.

\subsection*{Organization} Section \ref{sec:HpFIO} contains background on the Hardy spaces for FIOs, and Section \ref{sec:FIOs} concerns basic properties of FIOs. In Section \ref{sec:main} we prove our main result, and in Section \ref{sec:sharpness} we show that this result is essentially sharp.

\subsection*{Notation}

The natural numbers are $\N=\{1,2,\ldots\}$, and $\Z_{+}:=\N\cup\{0\}$. Throughout this article we fix $n\in\N$ with $n\geq2$.

For $\xi\in\Rn$ we write $\lb\xi\rb=(1+|\xi|^{2})^{1/2}$, and $\hat{\xi}=\xi/|\xi|$ if $\xi\neq0$. We use multi-index notation, where $\partial_{\xi}=(\partial_{\xi_{1}},\ldots,\partial_{\xi_{n}})$ and $\partial^{\alpha}_{\xi}=\partial^{\alpha_{1}}_{\xi_{1}}\ldots\partial^{\alpha_{n}}_{\xi_{n}}$  
for $\xi=(\xi_{1},\ldots,\xi_{n})\in\Rn$ and $\alpha=(\alpha_{1},\ldots,\alpha_{n})\in\Z_{+}^{n}$. Moreover, $\partial_{x\eta}^{2}\Phi$ is the mixed Hessian of a function $\Phi$ of the variables $x$ and $\eta$.
The Fourier transform of $f\in\Sw'(\Rn)$ is denoted by $\F f$ or $\widehat{f}$, and the Fourier multiplier with symbol $\ph\in\Sw'(\Rn)$ is denoted by $\ph(D)$. 

We write $f(s)\lesssim g(s)$ to indicate that $f(s)\leq Cg(s)$ for all $s$ and a constant $C>0$ independent of $s$, and similarly for $f(s)\gtrsim g(s)$ and $g(s)\eqsim f(s)$.

\section{Hardy spaces for Fourier integral operators}\label{sec:HpFIO}

In this section we first introduce the Hardy spaces for Fourier integral operators, and then we collect a few of their basic properties.

\subsection{Definitions}\label{subsec:HpFIOdef}

Fix a non-negative radial $\ph\in C^{\infty}_{c}(\Rn)$ such that $\ph(\xi)=0$ for $|\xi|>1$, and $\ph\equiv1$ in a neighborhood of zero. For $\w\in S^{n-1}$, $\sigma>0$ and $\xi\in\Rn\setminus\{0\}$, set $\ph_{\w,\sigma}(\xi):=c_{\sigma}\ph(\tfrac{\hat{\xi}-\w}{\sqrt{\sigma}})$, where $c_{\sigma}:=(\int_{S^{n-1}}\ph(\frac{e_{1}-\nu}{\sqrt{\sigma}})^{2}\ud\nu)^{-1/2}$ for $e_{1}=(1,0,\ldots,0)$ the first basis vector of $\Rn$ (this choice is irrelevant). Here and throughout, $\ud\nu$ is the normalized surface measure on the unit sphere $S^{n-1}$. Also set $\ph_{\w,\sigma}(0):=0$. Let $\Psi\in C^{\infty}_{c}(\Rn)$ be non-negative, radial, such that $\Psi(\xi)=0$ if $|\xi|\notin[1/2,2]$, and such that $\int_{0}^{\infty}\Psi(\sigma\xi)^{2}\frac{\ud\sigma}{\sigma}=1$ for all $\xi\neq 0$. Now, for $\w\in S^{n-1}$ and $\xi\in\Rn$, set 
\[
\ph_{\w}(\xi):=\int_{0}^{4}\Psi(\sigma\xi)\ph_{\w,\sigma}(\xi)\frac{\ud\sigma}{\sigma}.
\]
Some properties of these functions are as follows (see \cite[Remark 3.3]{Rozendaal21}):
\begin{enumerate}
\item\label{it:phiproperties1} For all $\w\in S^{n-1}$ and $\xi\neq0$ one has 
\begin{equation}\label{eq:phiwsupport}
\ph_{\w}(\xi)=0\text{ if }|\xi|<\tfrac{1}{8}\text{ or }|\hat{\xi}-\w|>2|\xi|^{-1/2}.
\end{equation}
\item\label{it:phiproperties2} For all $\alpha\in\Z_{+}^{n}$ and $\beta\in\Z_{+}$ there exists a $C_{\alpha,\beta}\geq0$ such that
\[
|(\w\cdot \partial_{\xi})^{\beta}\partial^{\alpha}_{\xi}\ph_{\w}(\xi)|\leq C_{\alpha,\beta}|\xi|^{\frac{n-1}{4}-\frac{|\alpha|}{2}-\beta}
\]
for all $\w\in S^{n-1}$ and $\xi\neq0$.
\item\label{it:phiproperties3} For all $\alpha\in\Z_{+}^{n}$ there exists a $C_{\alpha}\geq0$ such that
\[
\Big|\partial_{\xi}^{\alpha}\Big(\int_{S^{n-1}}\ph_{\w}(\xi)\ud\w\Big)^{-1}\Big|\leq C_{\alpha} |\xi|^{\frac{n-1}{4}-|\alpha|}
\]
for all $\xi\in\Rn$ with $|\xi|\geq1/2$. Hence there exists a radial $m\in S^{(n-1)/4}(\Rn)$ such that if $f\in\Sw'(\Rn)$ satisfies $\supp(\wh{f}\,)\subseteq \{\xi\in\Rn\mid |\xi|\geq1/2\}$, then 
\[
f=\int_{S^{n-1}}m(D)\ph_{\nu}(D)f\ud\nu.
\]
\end{enumerate}

In \eqref{it:phiproperties3}, $S^{(n-1)/4}(\Rn)$ is the standard class of symbols $a$ of order $(n-1)/4$. That is, $a\in C^{\infty}(\Rn)$, and for all $\alpha\in\Z_{+}^{n}$ there exists a $C_{\alpha}\geq 0$ such that
\[
|\partial_{\eta}^{\alpha}a(\eta)|\leq C_{\alpha}\lb \eta\rb^{(n-1)/4-|\alpha|}
\]
for all $\eta\in\Rn$.

Throughout, for simplicity of notation, we write $\HT^{p}(\Rn)=L^{p}(\Rn)$ for $1<p<\infty$, and $\HT^{\infty}(\Rn)=\bmo(\Rn)$. Moreover, $\HT^{1}(\Rn)$ is the classical local Hardy space. We also fix a $q\in C^{\infty}_{c}(\Rn)$ such that $q(\xi)=1$ for $|\xi|\leq 2$.

\begin{definition}\label{def:HpFIO}
For $p\in[1,\infty)$, $\Hp$ consists of those $f\in\Sw'(\Rn)$ such that $q(D)f\in L^{p}(\Rn)$, $\ph_{\w}(D)f\in \HT^{p}(\Rn)$ for almost all $\w\in S^{n-1}$, and
\[
\Big(\int_{S^{n-1}}\|\ph_{\w}(D)f\|_{\HT^{p}(\Rn)}^{p}\ud\w\Big)^{1/p}<\infty,
\]
endowed with the norm
\[
\|f\|_{\HT^{p}_{FIO}(\Rn)}:=\|q(D)f\|_{L^{p}(\Rn)}+\Big(\int_{S^{n-1}}\|\ph_{\w}(D)f\|_{\HT^{p}(\Rn)}^{p}\ud\w\Big)^{1/p}.
\]
Moreover, $\HT^{\infty}_{FIO}(\Rn):=(\HT^{1}_{FIO}(\Rn))^{*}$. For $p\in[1,\infty]$ and $s\in\R$, $\HT^{s,p}_{FIO}(\Rn)$ consists of all $f\in\Sw'(\Rn)$ such that $\lb D\rb^{s}f\in\HT^{p}_{FIO}(\Rn)$, endowed with the norm
\[
\|f\|_{\HT^{s,p}_{FIO}(\Rn)}:=\|\lb D\rb^{s}f\|_{\Hp}.
\] 
\end{definition}

In fact, this is not how the Hardy spaces for FIOs were originally defined in \cite{Smith98a} and \cite{HaPoRo20}, namely in terms of conical square functions and tent spaces over the cosphere bundle. That these definitions are equivalent was shown in \cite{FaLiRoSo19,Rozendaal21}.

\subsection{Properties}\label{subsec:HpFIOprop}

First recall the Sobolev embeddings from \eqref{eq:Sobolev}. By \cite[Theorem 7.4]{HaPoRo20}, these in fact hold for all $p\in[1,\infty]$, upon replacing $W^{s,p}(\Rn)$ by $\lb D\rb^{-s}\HT^{p}(\Rn)$.

The Hardy spaces for FIOs form a complex interpolation scale. That is, let $p_{0},p_{1},p\in[1,\infty]$, $s_{0},s_{1},s\in\R$ and $\theta\in[0,1]$ be such that $\frac{1}{p}=\frac{1-\theta}{p_{0}}+\frac{\theta}{p_{1}}$ and $s=(1-\theta)s_{0}+\theta s_{1}$. Then, as follows from \cite[Proposition 6.7]{HaPoRo20}, one has
\[
[\HT^{s_{0},p_{0}}_{FIO}(\Rn),\HT^{s_{1},p_{1}}_{FIO}(\Rn)]_{\theta}=\Hps.
\]
Moreover, by \cite[Proposition 6.8]{HaPoRo20}, for all $p\in[1,\infty)$ and $s\in\R$ one has
\[
(\Hps)^{*}=\HT^{-s,p'}_{FIO}(\Rn).
\]
Here the duality pairing is the standard distributional pairing $\lb f,g\rb$ for $f\in\Sw(\Rn)\subseteq \Hps$ and $g\in\HT^{-s,p'}_{FIO}(\Rn)\subseteq\Sw'(\Rn)$. By \cite[Proposition 6.6]{HaPoRo20}, the Schwartz functions lie dense in $\Hps$, so this uniquely determines the duality pairing.

The following lemma concerns a key feature of the Hardy spaces for FIOs: the relationship between the $\Hp$ norm and the $\HT^{p}(\Rn)$ norm of a function with frequency support in a dyadic-parabolic region.

\begin{lemma}\label{lem:dyadicpar}
Let $p\in[1,\infty]$ and $s\in\R$. Then there exists a $C>0$ such that the following statements hold for all $f\in\Hps$ such that 
\begin{equation}\label{eq:supportwp}
\supp(\wh{f}\,)\subseteq\{\xi\in\Rn\mid |\xi|\in [R/2,2R], |\hat{\xi}-\nu|\leq 2R^{-1/2}\}
\end{equation}
for some $R\geq1$ and $\nu\in S^{n-1}$:
\begin{enumerate}
\item\label{it:sobeq1}
If $p\leq 2$, then
\[
\frac{1}{C}R^{s-s(p)}\|f\|_{L^{p}(\Rn)} \leq \|f\|_{\Hps}\leq CR^{s-s(p)} \|f\|_{L^{p}(\Rn)}.
\]
\item\label{it:sobeq2}
If $p>2$, then
\[
\frac{1}{C}R^{s+s(p)}\|f\|_{L^{p}(\Rn)} \leq \|f\|_{\Hps}\leq CR^{s+s(p)} \|f\|_{L^{p}(\Rn)}.
\]
\end{enumerate}
\end{lemma}
\begin{proof}
Since $f$ has frequency support in a dyadic annulus, one has 
\[
\|\lb D\rb^{s}f\|_{\HT^{p}(\Rn)}\eqsim R^{s}\|f\|_{\HT^{p}(\Rn)}\eqsim R^{s}\|f\|_{L^{p}(\Rn)}.
\]
Hence the statement is a direct consequence of \cite[Proposition 6.4]{FaLiRoSo19}.
\end{proof}

\section{Fourier integral operators}\label{sec:FIOs}

For the general theory of Fourier integral operators, and the associated notions from symplectic geometry, we refer to \cite{Hormander09,Duistermaat11,Sogge17}. To prove Theorem \ref{thm:localsmooth} we will only use Corollary \ref{cor:halfwave}, concerning the operators $e^{it\phi(D)}$. However, these operators are part of a larger class of Fourier integral operators, introduced in Definition \ref{def:operator}, which are treated by the same methods. One of the key properties of the Hardy spaces for FIOs is their invariance under such operators.

\subsection{Definitions}\label{subsec:FIOdef}

For $m\in\R$, recall that $S^{m}_{1,0}$ consists of all $a\in C^{\infty}(\R^{2n})$ such that 
\begin{equation}\label{eq:seminorms}
\|a\|_{\alpha,\beta}:=\sup_{(x,\eta)\in\R^{2n}}\lb\eta\rb^{-m+|\alpha|}|\partial_{x}^{\beta}\partial_{\eta}^{\alpha}a(x,\eta)|<\infty
\end{equation}
for all $\alpha,\beta\in\Z_{+}^{n}$. In \cite{HaPoRo20} a larger symbol class was considered, the elements of which behave in an anisotropic manner which matches the dyadic-parabolic decomposition of $\Hp$. This class, denoted by $S^{m}_{1/2,1/2,1}$, contains H\"{o}rmander's $S^{m}_{1,1/2}$ class but is strictly contained in $S^{m}_{1/2,1/2}$. The statements in Proposition \ref{prop:operator} \eqref{it:operator2} and Lemma \ref{lem:operatornorm} below in fact extend to this larger symbol class.
 
In the following definition, a special case of \cite[Definition 2.11]{HaPoRo20}, and throughout, $o:=\Rn\times\{0\}$ denotes the zero section.

\begin{definition}\label{def:operator}
Let $m\in\R$, $a\in S^{m}_{1,0}$ and $\Phi\in C^{\infty}(\R^{2n}\setminus o)$, and set
\begin{equation}\label{eq:oscint}
Tf(x):=\int_{\Rn}e^{i\Phi(x,\eta)}a(x,\eta)\wh{f}(\eta)\ud\eta
\end{equation}
for $f\in\Sw(\Rn)$ and $x\in \Rn$. Then $T$ is a Fourier integral operator of order $m$ in \emph{standard form}, associated with a \emph{global canonical graph}, if:
\begin{enumerate}
\item\label{it:phase1} $\Phi$ is real-valued and positively homogeneous of degree $1$ in the $\eta$-variable;
\item\label{it:phase2} $\sup_{(x,\eta)\in \R^{2n}\setminus o}|\partial_{x}^{\beta}\partial_{\eta}^{\alpha}\Phi(x,\hat{\eta})|<\infty$ for all $\alpha,\beta\in\Z_{+}^{n}$ with $|\alpha|+|\beta|\geq 2$;
\item\label{it:phase3} $\inf_{(x,\eta)\in \R^{2n}\setminus o}| \det \partial^2_{x \eta} \Phi (x,\eta)|>0$;
\item\label{it:phase4} For each $x\in\Rn$, $\eta\mapsto \partial_{x}\Phi(x,\eta)$ is a bijection on $\R^{n}\setminus \{0\}$.
\end{enumerate}
\end{definition}

\begin{remark}\label{rem:oscint} 
By Hadamard's global inverse function theorem \cite[Theorem 6.2.8]{Krantz-Parks13}, condition \eqref{it:phase4} is superfluous for $n\geq3$. Moreover, \eqref{it:phase4} holds if $\Phi(x,\eta)=x\cdot\eta+\phi(\eta)$ for some $\phi\in C^{\infty}(\Rn\setminus\{0\})$ which is positively homogeneous of degree $1$. 

If \eqref{it:phase4} holds, then it follows from the global inverse function theorem that the map $(\partial_{\eta}\Phi(x,\eta),\eta)\mapsto (x,\partial_{x}\Phi(x,\eta))$ is a homogeneous canonical transformation on $\R^{2n}\setminus o$, and the canonical relation of $T$ is the graph of this transformation.
\end{remark}

Recall that a Fourier integral operator of order $m$, associated with a local canonical graph and having a compact Schwartz kernel, can, modulo an operator with a Schwartz kernel which is a Schwartz function, be expressed as a finite sum of operators which in appropriate coordinate systems are as in \eqref{eq:oscint}, where the symbol $a$ has compact support in the $x$-variable (see e.g. \cite[Proposition 6.2.4]{Sogge17}). In this case, \eqref{it:phase2} is automatically satisfied, \eqref{it:phase3} holds on the support of $a$, and the map in \eqref{it:phase4} is a locally defined homogeneous canonical transformation. By contrast, in Definition \ref{def:operator}, the symbols are not required to have compact spatial support, but the conditions on the phase function are required to hold on all of $\R^{2n}\setminus o$.

\subsection{Properties}\label{subsec:FIOprop}

The following proposition lists some classes of operators which are bounded on $\Hps$ but in general not on $W^{s,p}(\Rn)$.

\begin{proposition}\label{prop:operator}
Let $p\in[1,\infty]$, $s\in\R$ and $T$ satisfy one of the following conditions:
\begin{enumerate}
\item\label{it:operator1} $T$ is a Fourier integral operator of order zero, associated with a local canonical graph, with a compactly supported Schwartz kernel; 
\item\label{it:operator2} $T$ is a Fourier integral operator of order zero in standard form, associated with a global canonical graph;
\item\label{it:operator3} $T=\cos(t_{0}\sqrt{L})$, for $t_{0}\in\R$ and $L$ an elliptic divergence-form operator with bounded real-valued $C^{1,1}(\Rn)$ coefficients. Moreover, one has $1<p<\infty$, $2s(p)<1$, and $-2+s(p)\leq s\leq 2-s(p)$. 
\end{enumerate} 
Then $T:\Hps\to\Hps$ is bounded.
\end{proposition}
\begin{proof}
\eqref{it:operator1}: This follows in the same way as the case $s=0$ from \cite{HaPoRo20}; we will use the same terminology as there. By \cite[Proposition 2.14]{HaPoRo20}, there exist an $m\in\N$, normal oscillatory integral operators $(T_{j,1})_{j=1}^{m}$ and $(T_{j,2})_{j=1}^{m}$, and a smoothing operator $R$, such that $T=\sum_{j=1}^{m}T_{j,1}T_{j,2}+R$. As in the proof of \cite[Corollary 5.4]{HaPoRo20}, when conjugated with a suitable isometric wave packet transform $W:L^{2}(\Rn)\to L^{2}(\Tp)$, the $T_{j,1}$ and $T_{j,2}$ can each be written as a sum of an operator which satisfies off-singularity bounds and a residual operator. The smoothing operator $R$ yields a residual operator $WRW^{*}$ as well, by \cite[Proposition 5.3]{HaPoRo20}. All these operators are bounded on tent spaces over the cosphere bundle, by \cite[Theorem 3.7 and Proposition 3.6]{HaPoRo20}, but for $s\neq0$ it is important to note that they are also bounded on the weighted tent spaces $T^{p}_{s}(\Sp)$ used in \cite{Hassell-Rozendaal20}. This is shown in the same way as \cite[Proposition 2.4]{Hassell-Rozendaal20}, where the canonical relations under consideration are canonical graphs on all of $\R^{2n}\setminus o$, which plays no role in that proof. Moreover, \cite[Proposition 3.4]{Hassell-Rozendaal20} says that $\|f\|_{\Hps}\eqsim \|Wf\|_{T^{p}_{s}(\Sp)}$. It follows from this that the $T_{j,1}$, $T_{j,2}$ and $R$ are all bounded on $\Hps$, and therefore so is $T$.

\eqref{it:operator2}: Let $q'\in C^{\infty}_{c}(\Rn)$ equal $1$ near zero, and set $T_{1}:=T(1-q')(D)$ and $T_{2}:=Tq'(D)$. As in \eqref{it:operator1}, $WT_{1}W^{*}$ satisfies off-singularity bounds, by \cite[Corollary 5.2]{HaPoRo20}. Then the same arguments as above show that $T_{1}$ is bounded on $\Hps$.

On the other hand, for the low-frequency term $T_{2}$, one can integrate by parts and apply Young's inequality as in the proof of \cite[Theorem 1.18]{DosSantosFerreira-Staubach14} to show that $\partial_{x}^{\alpha}T_{2}$ is bounded on $L^{p}(\Rn)$ for each $\alpha\in\Z_{+}^{n}$. In particular, $T_{2}:L^{p}(\Rn)\to W^{N,p}(\Rn)$ for each $N\in\N$. Now choose $N>s+s(p)$, write $q'=q'q''$ for some $q''\in C^{\infty}_{c}(\Rn)$, so that $q''(D):\Hps\to L^{p}(\Rn)$, and use \eqref{eq:Sobolev}, to obtain
\[
T_{2}=T_{2}q''(D):\Hps\to W^{N,p}(\Rn)\subseteq \Hps.
\]

\eqref{it:operator3}: This is contained in \cite[Corollary 5.7]{Hassell-Rozendaal20}.
\end{proof}

\begin{remark}\label{rem:operators}
If the coefficients of the operator $L$ in \eqref{it:operator3} have additional regularity, then the conditions on $p$ and $s$ can be weakened, cf.~\cite{Hassell-Rozendaal20}. Moreover, similar statements hold for the other solution operator to the Cauchy problem associated with the wave equation $(\partial_{t}^{2}+L)u(t,x)=0$, and by \cite{Rozendaal20,Rozendaal22} pseudodifferential operators with rough coefficients are also bounded on $\Hps$ for suitable $p$ and $s$.
\end{remark}

In fact, we will need a more precise statement about the operator norm of families of Fourier integral operators as in \eqref{it:operator2}. 

\begin{lemma}\label{lem:operatornorm}
For each $t\in\R$, let $T_{t}$ be a Fourier integral operator of order zero in standard form, associated with a global canonical graph, with symbol $a_{t}$ and phase function $\Phi_{t}$. Suppose that, for all $\alpha,\beta\in\Z_{+}^{n}$, there exist $\veps,\kappa,M>0$ such that the following conditions hold for all $t\in\R$:
\begin{itemize}
\item\label{it:operatornorm1} $\|a_{t}\|_{\alpha,\beta}\leq \kappa (1+|t|)^{M}$;
\item\label{it:operatornorm2} $\sup_{(x,\eta)\in \R^{2n}\setminus o}|\partial_{x}^{\beta}\partial_{\eta}^{\alpha}\Phi_{t}(x,\hat{\eta})|\leq \kappa(1+|t|)^{M}$ if $|\alpha|+|\beta|\geq 2$;
\item\label{it:operatornorm3} $\inf_{(x,\eta)\in \R^{2n}\setminus o}| \det \partial^2_{x \eta} \Phi (x,\eta)|\geq \veps$.
\end{itemize} 
Let $p\in[1,\infty]$ and $s\in\R$. Then there exist $C,N>0$ such that
\[
\|T_{t}\|_{\La(\Hps)}\leq C(1+|t|)^{N}
\]
for all $t\in\R$.
\end{lemma}
\begin{proof}
This follows by keeping track of the constants in the proof of Proposition \ref{prop:operator} \eqref{it:operator2}. See also \cite[Remark 5.6]{HaPoRo20} and the remark after \cite[Theorem 6.10]{HaPoRo20}.
\end{proof}

We will apply this lemma to Fourier integral operators as in Remark \ref{rem:oscint}.

\begin{corollary}\label{cor:halfwave}
Let $p\in[1,\infty]$ and $s\in\R$, and let $\phi\in C^{\infty}(\Rn\setminus \{0\})$ be positively homogeneous of degree $1$. Then there exist $C,N>0$ such that
\[
\|e^{it\phi(D)}\|_{\La(\Hps)}\leq C(1+|t|)^{N}
\]
for all $t\in\R$.
\end{corollary}

\begin{remark}\label{rem:dyadic}
In the proof of our main result, we will only need bounds for $e^{it\phi(D)}$ when restricted to dyadic frequency shells. In this case one can rely directly on kernel estimates, and one does not need bounds on all of $\Hps$. On the other hand, the proof of Lemma \ref{lem:operatornorm} also relies on these kernel estimates (see also \cite[Remark 3.7]{Rozendaal21} for a simplification for the operators $e^{it\phi(D)}$). 
\end{remark}

\section{Main results}\label{sec:main}

Throughout this section, for each $k\in\Z_{+}$, fix a maximal collection $\Theta_{k}\subseteq S^{n-1}$ of unit vectors such that $|\nu-\nu'|\geq 2^{-k/2}$ for all $\nu,\nu'\in \Theta_{k}$. Note that $\Theta_{k}$ has approximately $2^{k(n-1)/2}$ elements. 
Let $(\chi_{\nu})_{\nu\in\Theta_{k}}\subseteq C^{\infty}(\Rn\setminus\{0\})$ be an associated partition of unity. That is, each $\chi_{\nu}$ is positively homogeneous of degree $0$ and satisfies $0\leq \chi_{\nu}\leq 1$ and 
\[
\supp(\chi_{\nu})\subseteq\{\xi\in\Rn\setminus\{0\}\mid |\hat{\xi}-\nu|\leq 2^{-k/2+1}\}.
\]
Moreover, $\sum_{\nu\in \Theta_{k}}\chi_{\nu}(\xi)=1$ for all $\xi\neq0$, and for all $\alpha\in\Z_{+}^{n}$ and $\beta\in\Z_{+}$ there exists a $C_{\alpha,\beta}\geq0$ independent of $k$ such that, if $2^{k-1}\leq |\xi|\leq 2^{k+1}$, then
\[
|(\hat{\xi}\cdot\partial_{\xi})^{\beta}\partial_{\xi}^{\alpha}\chi_{\nu}(\xi)|\leq C_{\alpha,\beta}2^{-k(|\alpha|/2+\beta)}
\]
for all $\nu\in \Theta_{k}$. Such a collection is straightforward to construct, in a similar manner as the wave packets in Section \ref{sec:HpFIO} (see \cite[Section IX.4]{Stein93}). Note also that the collection $\{\F^{-1}(\chi_{\nu})\mid k\in\Z_{+}, \nu\in \Theta_{k}\}$ is uniformly bounded in $L^{1}(\Rn)$.

\subsection{A discrete description of the $\Hp$ norm}

Using the decomposition of unity from above, we can give a discrete characterization of the $\Hp$ norm of a function with frequency support in a dyadic annulus.

\begin{proposition}\label{prop:equivalent}
Let $p\in[1,\infty]$ and $s\in\R$. Then there exists a $C>0$ such that the following holds. Let $f\in\Hps$ be such that $\supp(\wh{f}\,)\subseteq\{\xi\in\Rn\mid 2^{k-1}\leq |\xi|\leq 2^{k+1}\}$ for some $k\in\Z_{+}$. Then
\[
\frac{1}{C}\|f\|_{\Hps}\leq 2^{k(s+\frac{n-1}{2}(\frac{1}{2}-\frac{1}{p}))}\Big(\sum_{\nu\in\Theta_{k}}\|\chi_{\nu}(D)f\|_{L^{p}(\Rn)}^{p}\Big)^{1/p}\leq C\|f\|_{\Hps}
\]
for $p<\infty$, while for $p=\infty$ one has
\[
\frac{1}{C}\|f\|_{\HT^{s,\infty}_{FIO}(\Rn)}\leq 2^{k(s+\frac{n-1}{4})}\max_{\nu\in\Theta_{k}}\|\chi_{\nu}(D)f\|_{L^{\infty}(\Rn)}\leq C\|f\|_{\HT^{s,\infty}_{FIO}(\Rn)}.
\]
\end{proposition}
\begin{proof}
By a standard argument, using a Littlewood--Paley decomposition of $L^{p}(\Rn)$ (and thereby also of $\Hp$), we may suppose that $s=-\frac{n-1}{2}(\frac{1}{2}-\frac{1}{p})$. 

Note that the statement is true for $p=2$. Indeed, $\HT^{2}_{FIO}(\Rn)=L^{2}(\Rn)$ and 
\[
\Big(\sum_{\nu}\|\chi_{\nu}(D)f\|_{L^{2}(\Rn)}^{2}\Big)^{1/2}=\Big(\int_{\Rn}|\wh{f}(\xi)|^{2}\sum_{\nu}\chi_{\nu}(\xi)^{2}\ud\xi\Big)^{1/2}\eqsim \|f\|_{L^{2}(\Rn)}.
\]
In the final step we used that $1\lesssim \sum_{\nu}\chi_{\nu}(\xi)^{2}\leq (\sum_{\nu}\chi_{\nu}(\xi))^{2}=1$ for all $\xi\neq0$. More precisely, since there exists a $c_{n}\in\N$, independent of $k$ and $\Theta_{k}$, such that there are at most $c_{n}$ elements $\nu\in \Theta_{k}$ with $\chi_{\nu}(\xi)\neq 0$, one has $\sum_{\nu}\chi_{\nu}(\xi)^{2}\geq \max_{\nu}\chi_{\nu}(\xi)^{2}\geq 1/c_{n}^{2}$.

Hence, by interpolation and duality, it suffices to prove the statement for $p=1$. We may also choose $k$ large enough so that $q(D)f=0$, as follows e.g.~from the Sobolev embeddings in \eqref{eq:Sobolev}. Note that Lemma \ref{lem:dyadicpar} yields
\begin{equation}\label{eq:equivalent}
\|f\|_{\HT^{(n-1)/4,1}_{FIO}(\Rn)}\leq \sum_{\nu}\|\chi_{\nu}(D)f\|_{\HT^{(n-1)/4,1}_{FIO}(\Rn)}\eqsim \sum_{\nu}\|\chi_{\nu}(D)f\|_{L^{1}(\Rn)}.
\end{equation}
On the other hand,
\[
\sum_{\nu}\|\chi_{\nu}(D)f\|_{\HT^{(n-1)/4,1}_{FIO}(\Rn)}=\sum_{\nu}\int_{E_{\nu}}\|\chi_{\nu}(D)\lb D\rb^{(n-1)/4}\ph_{\w}(D)f\|_{\HT^{1}(\Rn)}\ud \w
\]
for $E_{\nu}:=\{\w\in S^{n-1}\mid |\w-\nu|\leq 2^{-k/2+3}\}$. Since the $\chi_{\nu}(D)$ have kernels that are uniformly in $L^{1}(\Rn)$, and because the $\nu\in \Theta_{k}$ are $2^{-k/2}$ separated, one now has
\begin{align*}
\sum_{\nu}\|\chi_{\nu}(D)f\|_{\HT^{(n-1)/4,1}_{FIO}(\Rn)}&\lesssim \sum_{\nu}\int_{E_{\nu}}\|\lb D\rb^{(n-1)/4}\ph_{\w}(D)f\|_{\HT^{1}(\Rn)}\ud \w\\
&\lesssim \|f\|_{\HT^{(n-1)/4,1}_{FIO}(\Rn)}.
\end{align*}
Together with \eqref{eq:equivalent}, this concludes the proof.
\end{proof}

By combining this proposition with Corollary \ref{cor:halfwave}, we can derive the following corollary, which will in turn play a key role in the proof of our main result.

\begin{corollary}\label{cor:wave}
Let $p\in[1,\infty]$, $s\in\R$ and $0\neq g\in\Sw(\R)$. Let $\phi\in C^{\infty}(\Rn\setminus \{0\})$ be positively homogeneous of degree $1$. Then there exists a $C>0$ such that the following holds. Let $f\in\Hps$ be such that $\supp(\wh{f}\,)\subseteq\{\xi\in\Rn\mid 2^{k-1}\leq |\xi|\leq 2^{k+1}\}$ for some $k\in\Z_{+}$. Then
\begin{align*}
\frac{1}{C}\|f\|_{\Hps}&\leq 2^{k(s+\frac{n-1}{2}(\frac{1}{2}-\frac{1}{p}))}\Big(\sum_{\nu\in\Theta_{k}}\int_{\R}\|\chi_{\nu}(D)g(t)e^{it\phi(D)}f\|_{L^{p}(\Rn)}^{p}\ud t\Big)^{1/p}\\
&\leq C\|f\|_{\Hps}
\end{align*}
for $p<\infty$, while for $p=\infty$ one has
\[
\frac{1}{C}\|f\|_{\HT^{s,\infty}_{FIO}(\Rn)}\leq 2^{k(s+\frac{n-1}{4})}\max_{\nu\in\Theta_{k}}\sup_{t\in\R}\|\chi_{\nu}(D)g(t)e^{it\phi(D)}f\|_{L^{\infty}(\Rn)}\leq C\|f\|_{\HT^{s,\infty}_{FIO}(\Rn)}.
\]
\end{corollary}
\begin{proof}
We consider $p<\infty$. The argument for $p=\infty$ is identical, up to a change in notation. By Proposition \ref{prop:equivalent}, one has
\begin{align*}
&2^{k(s+\frac{n-1}{2}(\frac{1}{2}-\frac{1}{p}))}\Big(\sum_{\nu\in\Theta_{k}}\int_{\R}\|\chi_{\nu}(D)g(t)e^{it\phi(D)}f\|_{L^{p}(\Rn)}^{p}\ud t\Big)^{1/p}\\
&\eqsim \Big(\int_{\R}\|g(t)e^{it\phi(D)}f\|_{\Hps}^{p}\ud t\Big)^{1/p}.
\end{align*}
Now Corollary \ref{cor:halfwave} yields
\begin{align*}
\Big(\int_{\R}\|g(t)e^{it\phi(D)}f\|_{\Hps}^{p}\ud t\Big)^{1/p}&\lesssim \sup_{t\in\R}\|(1+|t|)^{2}g(t)e^{it\phi(D)}f\|_{\Hps}\\
&\lesssim \|f\|_{\Hps}.
\end{align*}
On the other hand, 
\[
\|f\|_{\Hps}=\|e^{-it\phi(D)}e^{it\phi(D)}f\|_{\Hps}\lesssim \|g(t)e^{it\phi(D)}f\|_{\Hps}
\]
on any interval $I\subseteq\R$ such that $|g(t)|\gtrsim 1$ for all $t\in I$. It follows that
\[
\Big(\int_{\R}\|g(t)e^{it\phi(D)}f\|_{\Hps}^{p}\ud t\Big)^{1/p}\eqsim \|f\|_{\Hps}.\qedhere
\]
\end{proof}

\subsection{Main result}\label{subsec:mainresult}

Our main theorem concerns more general constant-coefficient operators than the Euclidean half-wave group $(e^{it\sqrt{-\Delta}})_{t\in\R}$. The generality of our result is restricted by the validity of the following $\ell^{p}$-decoupling inequality.

\begin{proposition}\label{prop:decoupling}
Let $p\in(2,\infty)$, and let $\phi\in C^{\infty}(\Rn\setminus\{0\})$ be positively homogeneous of degree $1$ and such that $\rank(\partial_{\xi\xi}^{2}\phi)\equiv n-1$. Let $g\in\Sw(\R)$ be such that $|g(t)|\geq1$ for $t\in[0,1]$, and $\wh{g}(\tau)=0$ for $\tau\notin[-1,1]$. Then, for each $\veps>0$, there exists a $C>0$ such that
\[
\Big(\int_{0}^{1}\|e^{it\phi(D)}f\|_{L^{p}(\Rn)}^{p}\ud t\Big)^{\frac{1}{p}}\leq C2^{k(d(p)+\veps)}\Big(\sum_{\nu\in\Theta_{k}}\int_{\R}\|\chi_{\nu}(D)g(t)e^{it\phi(D)}f\|_{L^{p}(\Rn)}^{p}\ud t\Big)^{\frac{1}{p}}
\]
for every $f\in L^{p}(\Rn)$ with $\supp(\wh{f}\,)\subseteq\{\xi\in\Rn\mid 2^{k-1}\leq |\xi|\leq 2^{k+1}\}$ for some $k\in\Z_{+}$.
\end{proposition}
\begin{proof}
First note that 
\[
\Big(\int_{0}^{1}\|e^{it\phi(D)}f\|_{L^{p}(\Rn)}^{p}\ud t\Big)^{1/p}\leq \Big(\int_{\R}\|g(t)e^{it\phi(D)}f\|_{L^{p}(\Rn)}^{p}\ud t\Big)^{1/p}
\]
and that the Fourier transform (in $n+1$ variables) of $g(t)e^{it\phi(D)}f(x)$ is supported in $\{(\xi,\tau)\mid |\tau-\phi(\xi)|\leq 1\}$. Moreover, the condition $\rank(\partial_{\xi\xi}^{2}\phi)\equiv n-1$ is equivalent to the condition that the cone $\{(\xi,\phi(\xi))\mid \xi\neq0\}$ has everywhere $n-1$ non-vanishing principal curvatures. Since $e^{it\phi(D)}f=\sum_{\nu\in\Theta_{k}}\chi_{\nu}(D)e^{it\phi(D)}f$, after rescaling to unit frequencies, the proposition follows from the $\ell^{p}$-decoupling inequality in \cite{Bourgain-Demeter17}, for hypersurfaces with non-vanishing Gaussian curvature, in a similar manner as how the $\ell^{2}$-decoupling inequality for the cone follows from the $\ell^{2}$-decoupling inequality for the paraboloid in \cite{Bourgain-Demeter15}. For more on this see \cite[Theorem 5.3]{BeHiSo21} and \cite{BeHiSo20}.
\end{proof}

We can now state and prove our main result, of which Theorem \ref{thm:localsmooth} is a corollary.

\begin{theorem}\label{thm:localsmoothmain}
Let $p\in(2,\infty)$ and $s\in\R$, and let $\phi\in C^{\infty}(\Rn\setminus\{0\})$ be positively homogeneous of degree $1$ and such that $\rank(\partial_{\xi\xi}^{2}\phi)\equiv n-1$. Then, for each $\veps>0$, there exists a $C>0$ such that 
\[
\Big(\int_{0}^{1}\|e^{it\phi(D)}f\|_{W^{s,p}(\Rn)}^{p}\ud t\Big)^{1/p}\leq C\|f\|_{\HT^{s+d(p)-s(p)+\veps,p}_{FIO}(\Rn)}
\]
for all $f\in \HT^{s+d(p)-s(p)+\veps,p}_{FIO}(\Rn)$.
\end{theorem}
\begin{proof}
We may consider the case where $s=0$. We also claim that we may suppose that $\supp(\wh{f}\,)\subseteq\{\xi\in\Rn\mid |\xi|\in[2^{k-1},2^{k+1}]\}$ for some $k\in\Z_{+}$, and show that
\begin{equation}\label{eq:toprove}
\Big(\int_{0}^{1}\|e^{it\phi(D)}f\|_{L^{p}(\Rn)}^{p}\ud t\Big)^{1/p}\lesssim \|f\|_{\HT^{d(p)-s(p)+\veps,p}_{FIO}(\Rn)}.
\end{equation}
Indeed, suppose that we have proved this, and let a general $f\in\HT^{d(p)-s(p)+\veps,p}_{FIO}(\Rn)$ be given. Let $q'\in C^{\infty}_{c}(\Rn)$ and $\kappa>0$ be such that $q(\xi)=0$ for $|\xi|> \kappa$, and $q'(\xi)=1$ for $|\xi|\leq \kappa$. For each $t\in(0,1)$, write $f=q'(tD)q'(D)q(D)f+(1-q)(D)f$. The kernels of $q'(D)e^{i\phi(D)}$ and $q'(D)\lb D\rb^{-d(p)+s(p)-\veps}$ are integrable, so 
\begin{align*}
&\Big(\int_{0}^{1}\|e^{it\phi(D)}q'(tD)q'(D)q(D)f\|_{L^{p}(\Rn)}^{p}\ud t\Big)^{1/p}\\
&\leq \sup_{t\in(0,1)}\|e^{it\phi(D)}q'(tD)q'(D)q(D)f\|_{L^{p}(\Rn)}\lesssim \|q'(D)q(D)f\|_{L^{p}(\Rn)}\\
&\lesssim \|\lb D\rb^{d(p)-s(p)+\veps}q(D)f\|_{L^{p}(\Rn)}\leq  \|f\|_{\HT^{d(p)-s(p)+\veps,p}_{FIO}(\Rn)}.
\end{align*}
On the other hand, since $q(\xi)=1$ for $|\xi|\leq 2$, we can use a standard Littlewood--Paley decomposition to write $(1-q)(D)f=\sum_{k=3}^{\infty}\psi(2^{-k+1}D)(1-q)(D)f$. Here $\psi\in C^{\infty}_{c}(\Rn)$ satisfies $\sum_{k=-\infty}^{\infty}\psi(2^{-k+1}\xi)=1$ for all $\xi\neq 0$, and $\psi(\xi)=0$ if $|\xi|\notin [1/2,2]$. The convolution kernels of the $\psi(2^{-k+1}D)(1-q)(D)$ form a uniformly bounded subset of $L^{1}(\Rn)$, so
\[
\|\psi(2^{-k+1}D)(1-q)(D)\lb D\rb^{d(p)-s(p)+\veps}f\|_{\HT^{p}_{FIO}(\Rn)}\lesssim \|\lb D\rb^{d(p)-s(p)+\veps}f\|_{\HT^{p}_{FIO}(\Rn)}
\]
for an implicit constant independent of $k\in\Z_{+}$ and $f$ (see \cite[Theorem 6.1]{FaLiRoSo19}). Moreover, for all $k\in\Z_{+}$, $r\in\R$ and $h\in\Hp$, one has 
\[
\|\psi(2^{-k+1}D)h\|_{\HT^{r,p}_{FIO}(\Rn)}\eqsim 2^{kr}\|\psi(2^{-k+1}D)h\|_{\Hp},
\]
as follows from the definition of $\Hp$ and from a  standard fact about the $L^{p}(\Rn)$ norm of functions with frequency support in a dyadic annulus.
Hence, using the assumption with $\veps$ replaced by $\veps/2$, we obtain
\begin{align*}
&\Big(\int_{0}^{1}\|e^{it\phi(D)}(1-q)(D)f\|_{L^{p}(\Rn)}^{p}\ud t\Big)^{1/p}\\
&\leq \sum_{k=3}^{\infty}\Big(\int_{0}^{1}\|e^{it\phi(D)}\psi(2^{-k+1}D)(1-q)(D)f\|_{L^{p}(\Rn)}^{p}\ud t\Big)^{1/p}\\
&\lesssim \sum_{k=3}^{\infty}\|\psi(2^{-k+1}D)(1-q)(D)f\|_{\HT^{d(p)-s(p)+\veps/2,p}_{FIO}(\Rn)}\\
&\eqsim \sum_{k=3}^{\infty}2^{-k\veps/2}\|\psi(2^{-k+1}D)(1-q)(D)f\|_{\HT^{d(p)-s(p)+\veps,p}_{FIO}(\Rn)}\\
&\lesssim \sum_{k=3}^{\infty}2^{-k\veps/2}\|f\|_{\HT^{d(p)-s(p)+\veps,p}_{FIO}(\Rn)}\lesssim \|f\|_{\HT^{d(p)-s(p)+\veps,p}_{FIO}(\Rn)}.
\end{align*}
This proves the claim, and it remains to establish \eqref{eq:toprove} whenever $\supp(\wh{f}\,)\subseteq\{\xi\in\Rn\mid |\xi|\in[2^{k-1},2^{k+1}]\}$ for some $k\in\Z_{+}$.

Let the partition of unity $(\chi_{\nu})_{\nu\in\Theta_{k}}$ be as above. Also let $g\in\Sw(\R)$ be such that $|g(t)|\geq1$ for $t\in[0,1]$, and $\wh{g}(\tau)=0$ for $\tau\notin[-1,1]$. Then Proposition \ref{prop:decoupling} yields
\begin{align*}
\Big(\int_{0}^{1}\|e^{it\phi(D)}f\|_{L^{p}(\Rn)}^{p}\ud t\Big)^{1/p}&\leq \Big(\int_{\R}\Big\|\sum_{\nu\in\Theta_{k}}g(t)\chi_{\nu}(D)e^{it\phi(D)}f\Big\|_{L^{p}(\Rn)}^{p}\ud t\Big)^{1/p}\\
&\lesssim 2^{k(d(p)+\veps)}\Big(\!\sum_{\nu\in\Theta_{k}}\int_{\R}\|\chi_{\nu}(D)g(t)e^{it\phi(D)}f\|_{L^{p}(\Rn)}^{p}\ud t\Big)^{1/p}
\end{align*}
for implicit constants independent of $f$ and $k$. By Corollary \ref{cor:wave}, the final quantity is equivalent to $\|f\|_{\HT^{d(p)-s(p)+\veps,p}_{FIO}(\Rn)}$, as required. 
\end{proof}

\subsection{Invariance of estimates}

As an immediate corollary of Theorem \ref{thm:localsmooth}, we obtain the invariance of local smoothing estimates under FIOs and related operators. Indeed, under the assumptions of Theorem \ref{thm:localsmoothmain}, if $T$ is a bounded operator on $\HT^{s+d(p)-s(p)+\veps,p}_{FIO}(\Rn)$, then Theorem \ref{thm:localsmooth} yields a $C'\geq0$ such that
\begin{equation}\label{eq:bddest}
\Big(\int_{0}^{1}\|e^{it\phi(D)}Tf\|_{W^{s,p}(\Rn)}^{p}\ud t\Big)^{1/p}\leq C'\|f\|_{\HT^{s+d(p)-s(p)+\veps,p}_{FIO}(\Rn)}
\end{equation}
for all $f\in \HT^{s+d(p)-s(p)+\veps,p}_{FIO}(\Rn)$. Now Proposition \ref{prop:operator} yields the following result.

\begin{corollary}\label{cor:bddest}
Let $p\in(2,\infty)$ and $s\in\R$, and let $\phi\in C^{\infty}(\Rn\setminus\{0\})$ be positively homogeneous of degree $1$ and such that $\rank(\partial_{\xi\xi}^{2}\phi)\equiv n-1$. Let $T$ be an operator as in Proposition \ref{prop:operator}, where in \eqref{it:operator3} we suppose instead that $-2+2s(p)\leq s+d(p)+\veps\leq 2$. Then, for each $\veps>0$, there exists a $C'>0$ such that \eqref{eq:bddest} holds for all $f\in \HT^{s+d(p)-s(p)+\veps,p}_{FIO}(\Rn)$.
\end{corollary}

In particular, by combining \eqref{eq:bddest} with \eqref{eq:Sobolev}, one sees that the local smoothing estimates in \eqref{eq:conj2} are invariant under application of the operators in Corollary \ref{cor:bddest}. 

\begin{remark}\label{rem:variable}
A variable-coefficient version of \eqref{eq:conj2} holds on compact manifolds \cite{BeHiSo20}. This implies in particular that \eqref{eq:conj2} is invariant under application of FIOs associated with a local canonical graph which have a compactly supported Schwartz kernel, as in Proposition \ref{prop:operator} \eqref{it:operator1}. However, by itself this does not imply the stronger estimate \eqref{eq:bddest} for such operators. In this article we will not explore variable-coefficient versions of Theorem \ref{thm:localsmoothmain}.
\end{remark}

\section{Sharpness of the results}\label{sec:sharpness}

In this section we show that the exponent $d(p)-s(p)$ in Theorems \ref{thm:localsmooth} and \ref{thm:localsmoothmain} is sharp. We also prove a proposition which deals with the case where $p\in[1,2]\cup\{\infty\}$. 

We need two lemmas. The first one relates certain FIOs to translation operators, and will be applied to functions with frequency support in a dyadic-parabolic piece. A similar statement holds for more general FIOs, using bicharacteristic flows on phase space (see \cite{Hassell-Rozendaal20}), but we will not need such generality here.

\begin{lemma}\label{lem:translation}
Let $\phi\in C^{\infty}(\Rn\setminus\{0\})$ be positively homogeneous of degree $1$, and let $h\in \Sw(\Rn)$ be such that $\supp(\wh{h})\subseteq\Rn\setminus\{0\}$ is compact. Let $\nu\in S^{n-1}$ and set 
\[
\kappa(\phi,h,\nu):=\sup\{|(\partial_{\xi}\phi(\hat{\xi})-\partial_{\xi}\phi(\nu))\cdot\xi|\mid \xi\in\supp(\wh{h})\}.
\]
 Then
\[
|e^{it\phi(D)}h(x)-h(x+t\partial_{\xi}\phi(\nu))|\leq \|\wh{h}\|_{L^{1}(\Rn)}\kappa(\phi,h,\nu)|t|
\]
for all $x\in\Rn$ and $t\in\R$ such that $\kappa(\phi,h,\nu)|t|\leq 1$.
\end{lemma}
\begin{proof}
We use the homogeneity of $\phi$ to write
\begin{align*}
&(2\pi)^{n}\big(e^{it\phi(D)}h(x)-h(x+t\partial_{\xi}\phi(\nu))\big)=\int_{\Rn}e^{ix\cdot\xi}(e^{it\phi(\xi)}-e^{it\partial_{\xi}\phi(\nu)\cdot\xi})\wh{h}(\xi)\ud\xi\\
&=\int_{\Rn}e^{ix\cdot\xi}e^{it\partial_{\xi}\phi(\nu)\cdot\xi}(e^{it(\partial_{\xi}\phi(\hat{\xi})-\partial_{\xi}\phi(\nu))\cdot\xi}-1)\wh{h}(\xi)\ud\xi.
\end{align*}
Now simply note that, in the last line, the term in brackets is bounded in absolute value by $|t|\kappa(\phi,h,\nu)e^{|t|\kappa(\phi,h,\nu)}$.
\end{proof}

Next, to prove sharpness for $2<p<2(n+1)/(n-1)$, we will use the following lemma, which relates the Hardy spaces for FIOs to square functions similar to those in \cite{GuWaZh20}. Throughout, we use notation as in Section \ref{sec:main}.

\begin{lemma}\label{lem:squarefun}
Let $\phi\in C^{\infty}(\Rn\setminus\{0\})$ be positively homogeneous of degree $1$, and let $p\in[1,\infty)$, $s\in\R$ and $C\geq0$ be such that
\begin{equation}\label{eq:assump}
\Big(\int_{0}^{1}\|e^{it\phi(D)}f\|_{L^{p}(\Rn)}^{p}\ud t\Big)^{1/p}\leq C\|f\|_{\Hps}
\end{equation}
for all $f\in\Hps$. Then there exists a $C'\geq0$ such that
\[
\Big(\int_{0}^{1}\int_{\Rn}\Big(\sum_{\nu\in\Theta_{k}}|\chi_{\nu}(D)e^{it\phi(D)}f(x)|^{2}\Big)^{p/2}\ud x\ud t\Big)^{1/p}\leq C'\|f\|_{\Hps}
\]
for all $f\in\Hps$ such that $\supp(\wh{f}\,)\subseteq \{\xi\in\Rn\mid 2^{k-1}\leq |\xi|\leq 2^{k+1}\}$ for some $k\in\Z_{+}$.
\end{lemma}
\begin{proof}
Let $(\veps_{\nu})_{\nu\in\Theta_{k}}$ be an independent sequence of Rademacher random variables, and for $x\in\Rn$ write 
\[
f_{\veps}(x):=\sum_{\nu\in\Theta_{k}}\veps_{\nu}\chi_{\nu}(D)f(x).
\]
Then Proposition \ref{prop:equivalent} shows that $f_{\veps}\in\Hps$ with $\|f_{\veps}\|_{\Hps}\eqsim \|f\|_{\Hps}$. Now one simply combines \eqref{eq:assump} with Khintchine's inequality.
\end{proof}

We can now state the main result of this section, the proof of which relies on examples that are typically used to show sharpness of decoupling inequalities (see e.g. \cite{GuWaZh20,Wolff00}). These examples were communicated to us by Po--Lam Yung.

\begin{theorem}\label{thm:sharp}
Let $\phi\in C^{\infty}(\Rn\setminus\{0\})$ be positively homogeneous of degree $1$. Let $p\in(2,\infty)$, $s\in\R$ and $C\geq0$ be such that
\begin{equation}\label{eq:sharp}
\Big(\int_{0}^{1}\|e^{it\phi(D)}f\|_{L^{p}(\Rn)}^{p}\ud t\Big)^{1/p}\leq C\|f\|_{\HT^{s,p}_{FIO}(\Rn)}
\end{equation}
for all $f\in \HT^{s,p}_{FIO}(\Rn)$. Then $s\geq d(p)-s(p)$.
\end{theorem}
\begin{proof}
We want to prove $s\geq d(p)-s(p)=\max(s(p)-\frac{1}{p},0)$. We first show that $s\geq s(p)-\frac{1}{p}$. To this end we apply Lemma \ref{lem:translation} in the case where the frequency support of the relevant function fills up a complete dyadic-parabolic piece.

More precisely, for $k\in\Z_{+}$, let $(f_{\nu})_{\nu\in\Theta_{k}}\subseteq \Sw(\Rn)$ be a collection with the following properties. For each $\nu\in\Theta_{k}$ one has $\chi_{\nu}(D)f_{\nu}=f_{\nu}$, and $f_{\nu}(x)\geq 1$ for all $x\in\Rn$ with $|x|\leq 2^{-k}$. Moreover, $\|\F(f_{\nu})\|_{L^{1}(\Rn)}\eqsim 1$ and $\|f_{\nu}\|_{L^{p}(\Rn)}\eqsim 2^{-k\frac{n+1}{2p}}$ for implicit constants independent of $\nu$ and $k$. To construct such a collection, one may start off with a function $\psi\in \Sw(\Rn)$ such that $\psi(x)\geq1$ whenever $|x|\leq 1$, and $\wh{\psi}(\xi)=0$ if $|\xi|>c$, for some small $c>0$. Then set
\[
f_{\nu}(x):=e^{i2^{k}\nu\cdot x}\psi(2^{k}\nu\cdot x+2^{k/2}\Pi_{\nu}^{\perp}x)
\] 
where $\Pi^{\perp}_{\nu}$ is the orthogonal projection onto the complement of the span of $\nu$. 

Next, let $f:=\sum_{\nu\in\Theta_{k}}f_{\nu}$. Then Proposition \ref{prop:equivalent} yields $f\in\Hps$ and
\begin{equation}\label{eq:HpFIObound}
\|f\|_{\Hps}\eqsim 2^{k(s+\frac{n-1}{2}(\frac{1}{2}-\frac{1}{p}))}\Big(\sum_{\nu\in\Theta_{k}}\|f_{\nu}\|_{L^{p}(\Rn)}^{p}\Big)^{1/p}\eqsim 2^{k(s+\frac{n-1}{2}(\frac{1}{2}-\frac{1}{p})-\frac{1}{p})}.
\end{equation}
Moreover, Lemma \ref{lem:translation} and the assumptions on $f_{\nu}$ yield 
\begin{equation}\label{eq:translatebound}
|e^{it\phi(D)}f_{\nu}(x)-f_{\nu}(x+t\partial_{\xi}\phi(\nu))|\lesssim  t2^{-k/2}
\end{equation}
whenever $0\leq t\lesssim 2^{-k/2}$, for implicit constants independent of $\nu$ and $k$. Also, $f_{\nu}(x+t\partial_{\xi}\phi(\nu))\geq1$ if $t\leq 2^{-k}$ and $|x|\leq c'2^{-k}$ for some small $c'>0$, since $f_{\nu}(x)\geq1$ for $|x|\leq 2^{-k}$ and because $\partial_{\xi}\phi$ is uniformly bounded on $S^{n-1}$. By combining this with \eqref{eq:translatebound}, we obtain
\[
e^{it\phi(D)}f_{\nu}(x)\geq f_{\nu}(x+t\partial_{\xi}\phi(\nu))-|e^{it\phi(D)}f_{\nu}(x)-f_{\nu}(x+t\partial_{\xi}\phi(\nu))|\gtrsim 1
\]
for $|x|\leq c'2^{-k}$ and $0\leq t\leq c''2^{-k}$ for some small enough $c''>0$. Now \eqref{eq:sharp} and \eqref{eq:HpFIObound} yield
\begin{align*}
2^{-k\frac{n+1}{p}}2^{k\frac{n-1}{2}}&\lesssim \Big(\int_{0}^{1}\Big\|\sum_{\nu\in\Theta_{k}}e^{it\phi(D)}f_{\nu}\Big\|_{L^{p}(\Rn)}^{p}\ud t\Big)^{1/p}=\Big(\int_{0}^{1}\|e^{it\phi(D)}f\|_{L^{p}(\Rn)}^{p}\ud t\Big)^{1/p}\\
&\lesssim \|f\|_{\HT^{s,p}_{FIO}(\Rn)}\eqsim 2^{k(s+\frac{n-1}{2}(\frac{1}{2}-\frac{1}{p})-\frac{1}{p})},
\end{align*}
which implies that $s\geq s(p)-\frac{1}{p}$.

Next, we show that $s\geq 0$. Here we apply Lemma \ref{lem:translation} to functions with frequency support of unit size in a dyadic-parabolic piece, and we also rely on Lemma \ref{lem:squarefun}. Let $\psi\in\Sw(\Rn)$ be as before, and for $k\in\Z_{+}$, $\nu\in\Theta_{k}$ and $x\in\Rn$ set
\[
g_{\nu}(x):=e^{i2^{k}\nu\cdot x}\psi(x).
\]
Since $\wh{\psi}(\xi)=0$ if $|\xi|>c$, for some small $c>0$, we may again assume that $\chi_{\nu}(D)g_{\nu}=g_{\nu}$. Also, $\|g_{\nu}\|_{L^{p}(\Rn)}\eqsim 1\eqsim \|\F(g_{\nu})\|_{L^{1}(\Rn)}$. 

Write $g:=\sum_{\nu\in\Theta_{k}}g_{\nu}$. Then $g\in\Hps$ with
\begin{equation}\label{eq:HpFIObound2}
\|g\|_{\Hps}\eqsim 2^{k(s+\frac{n-1}{2}(\frac{1}{2}-\frac{1}{p}))}\Big(\sum_{\nu\in\Theta_{k}}\|g_{\nu}\|_{L^{p}(\Rn)}^{p}\Big)^{1/p}\eqsim 2^{k(s+\frac{n-1}{4})},
\end{equation}
by Proposition \ref{prop:equivalent}. Moreover, since $\supp(\F g_{\nu})\subseteq \supp(\chi_{\nu})$ is of unit size, Lemma \ref{lem:translation} yields
\[
\big|e^{it\phi(D)}g_{\nu}(x)-g_{\nu}(x+t\partial_{\xi}\phi(\nu))\big|\lesssim  t
\]
whenever $0\leq t\leq \kappa$, for some small $\kappa>0$ independent of $k$ and $\nu$. As before, one also has $g_{\nu}(x+t\partial_{\xi}\phi(\nu))\geq 1$ whenever $|x|\lesssim 1$ and $0\leq t\lesssim 1$ are small enough. Hence
\[
e^{it\phi(D)}g_{\nu}(x)\geq g_{\nu}(x+t\partial_{\xi}\phi(\nu))-\big|e^{it\phi(D)}g_{\nu}(x)-g_{\nu}(x+t\partial_{\xi}\phi(\nu))\big|\gtrsim 1
\]
for $|x|\leq \kappa'$ and $0\leq t\leq \kappa''$, for small enough $\kappa',\kappa''>0$. Now combine this with \eqref{eq:sharp}, Lemma \ref{lem:squarefun} and \eqref{eq:HpFIObound2} to obtain
\begin{align*}
2^{k\frac{n-1}{4}}&\lesssim \Big(\int_{0}^{1}\int_{\Rn}\Big(\sum_{\nu\in\Theta_{k}}|e^{it\phi(D)}g_{\nu}(x)|^{2}\Big)^{p/2}\ud x\ud t\Big)^{1/p}\\
&=\Big(\int_{0}^{1}\int_{\Rn}\Big(\sum_{\nu\in\Theta_{k}}|\chi_{\nu}(D)e^{it\phi(D)}g(x)|^{2}\Big)^{p/2}\ud x\ud t\Big)^{1/p}\lesssim \|g\|_{\Hps}
\\
&\eqsim 2^{k(s+\frac{n-1}{4})},
\end{align*}
which implies that $s\geq0$.
\end{proof}

\begin{remark}\label{rem:decoupling}
Recall that Theorem \ref{thm:localsmoothmain} relied on the $\ell^{p}$-decoupling inequality from \cite{Bourgain-Demeter17}, for the hypersurface associated with $(e^{it\phi(D)})_{t\in\R}$, in Proposition \ref{prop:decoupling}. As a byproduct, Theorem \ref{thm:sharp} shows that each of these decoupling inequalities is sharp.
\end{remark}

\begin{remark}\label{rem:nosmooth}
To prove the inequality $s\geq s(p)-\frac{1}{p}$ in Theorem \ref{thm:sharp}, i.e.~deal with the case where $p\geq 2(n+1)/(n-1)$, we cannot simply rely on the fact that the local smoothing conjecture is sharp. Indeed, while such an approach suffices for $(e^{it\sqrt{-\Delta}})_{t\in\R}$, or more generally under the conditions of Theorem \ref{thm:localsmoothmain}, Theorem \ref{thm:sharp} does not make any assumptions on the rank of $\partial_{\xi\xi}^{2}\phi$, and the $L^{p}$-mapping properties of $e^{it\phi(D)}$ improve when $\rank(\partial_{\xi\xi}^{2}\phi)$ is small (see \cite{SeSoSt91}). See also Remark \ref{rem:nodisp}.
\end{remark}

For completeness we also include the following proposition, for $p\notin (2,\infty)$.

\begin{proposition}\label{prop:otherp}
Let $\phi\in C^{\infty}(\Rn\setminus\{0\})$ be positively homogeneous of degree $1$. Let $p\in[1,2]\cup \{\infty\}$ and $s\in\R$. Then the following are equivalent:
\begin{enumerate}
\item\label{it:otherp1} One has $[t\mapsto e^{it\phi(D)}f]\in L^{p}([0,1];\HT^{p}(\Rn))$ for all $f\in \HT^{s,p}_{FIO}(\Rn)$;
\item\label{it:otherp2} $s\geq s(p)$.
\end{enumerate}  
\end{proposition}
\begin{proof}
By Corollary \ref{cor:wave} and the Sobolev embeddings in \eqref{eq:Sobolev}, one has
\begin{equation}\label{eq:trivial}
[t\mapsto e^{it\phi(D)}]\in C(\R;\HT^{s(p),p}_{FIO}(\Rn))\subseteq L^{p}([0,1];\HT^{p}(\Rn))
\end{equation}
for all $1\leq p\leq \infty$ and $f\in\HT^{s(p),p}_{FIO}(\Rn)$, which proves the implication \eqref{it:otherp2}$\Rightarrow$\eqref{it:otherp1}. 

For $p=\infty$, the converse implication follows directly from the first part of the proof of Theorem \ref{thm:sharp}, up to a change of notation. For $p\leq 2$, let $f\in\Hps$ be such that
\[
\supp(\wh{f}\,)\subseteq\{\xi\in\Rn\mid 2^{k-1}\leq |\xi|\leq 2^{k+1}, |\hat{\xi}-\w|\leq 2^{-k/2+1}\}
\]
for some $k\in\Z_{+}$ and $\w\in S^{n-1}$. Then Corollary \ref{cor:wave} and Lemma \ref{lem:dyadicpar} yield
\[
2^{k(s(p)-s)}\|f\|_{\Hps}\eqsim 2^{k(s(p)-s)}\|e^{i\phi(D)}f\|_{\Hps}\eqsim \|e^{it\phi(D)}f\|_{\HT^{p}(\Rn)},
\]
for any $t\in[0,1]$. Now integrate this in time and apply \eqref{it:otherp1} to obtain
\[
2^{k(s(p)-s)}\|f\|_{\Hps}\eqsim \Big(\int_{0}^{1}\|e^{it\phi(D)}f\|_{\HT^{p}(\Rn)}^{p}\ud t\Big)^{1/p}\lesssim \|f\|_{\Hps},
\]
so that $s\geq s(p)$.
\end{proof}

As in Remark \ref{rem:nosmooth}, Proposition \ref{prop:otherp} does not follow in full generality from the fact that no local smoothing can occur for $p\notin (2,\infty)$.

\begin{remark}\label{rem:nodisp}
We briefly summarize what we have shown. Let $1\leq p\leq \infty$ and $s\in\R$, and let $\phi\in C^{\infty}(\Rn\setminus\{0\})$ be positively homogeneous of degree $1$. By \eqref{eq:trivial}, one has
\begin{equation}\label{eq:summarize}
[t\mapsto e^{it\phi(D)}f]\in L^{p}([0,1];\HT^{p}(\Rn))\text{ for all } f\in\Hps
\end{equation}
whenever $s\geq s(p)$. Because $(e^{it\phi(D)})_{t\in\R}$ is a group, there can be no improvement of the index in the first inclusion in \eqref{eq:trivial}, not even at a single time. Moreover, the condition $s\geq s(p)$ is also necessary for \eqref{eq:summarize} in some situations, for example when $p\in[1,2]\cup\{\infty\}$ or when $\phi\equiv0$, by Proposition \ref{prop:otherp} and because the Sobolev embeddings in \eqref{eq:Sobolev} are sharp. On the other hand, for $p\in(2,\infty)$, the condition $s\geq d(p)-s(p)$ is necessary for any $\phi$, while $s>d(p)-s(p)$ suffices when $\rank(\partial_{\xi\xi}^{2}\phi)$ is maximal everywhere. 
\end{remark}

\section*{Acknowledgments}

This work arose from conversations with Po-Lam Yung, who made various suggestions that helped improve it significantly. The author is very grateful to him for sharing his ideas, without which this article would not have been written. The author would also like to thank Robert Schippa and the anonymous referee for useful comments.

\bibliographystyle{plain}
\bibliography{Bibliography}

\begin{thebibliography}{10}

\bibitem{BeHiSo20}
D.~Beltran, J.~Hickman, and C.~Sogge.
\newblock Variable coefficient {W}olff-type inequalities and sharp local
  smoothing estimates for wave equations on manifolds.
\newblock {\em Anal. PDE}, 13(2):403--433, 2020.

\bibitem{BeHiSo21}
D.~Beltran, J.~Hickman, and C.~Sogge.
\newblock Sharp local smoothing estimates for {F}ourier integral operators.
\newblock In {\em Geometric aspects of harmonic analysis}, volume~45 of {\em
  Springer INdAM Ser.}, pages 29--105. Springer, Cham, 2021.

\bibitem{Bourgain-Demeter15}
J.~Bourgain and C.~Demeter.
\newblock The proof of the {$l^2$} decoupling conjecture.
\newblock {\em Ann. of Math. (2)}, 182(1):351--389, 2015.

\bibitem{Bourgain-Demeter17}
J.~Bourgain and C.~Demeter.
\newblock Decouplings for curves and hypersurfaces with nonzero {G}aussian
  curvature.
\newblock {\em J. Anal. Math.}, 133:279--311, 2017.

\bibitem{DosSantosFerreira-Staubach14}
D.~Dos Santos~Ferreira and W.~Staubach.
\newblock Global and local regularity of {F}ourier integral operators on
  weighted and unweighted spaces.
\newblock {\em Mem. Amer. Math. Soc.}, 229(1074):xiv+65, 2014.

\bibitem{Duistermaat11}
J.~Duistermaat.
\newblock {\em Fourier integral operators}.
\newblock Modern Birkh\"{a}user Classics. Birkh\"{a}user/Springer, New York,
  2011.
\newblock Reprint of the 1996 edition, based on the original lecture notes
  published in 1973.

\bibitem{FaLiRoSo19}
Z.~Fan, N.~Liu, J.~Rozendaal, and L.~Song.
\newblock Characterizations of the {H}ardy space
  $\mathcal{H}^{1}_{FIO}(\mathbb{R}^{n})$ for {F}ourier integral operators.
\newblock To appear in {S}tudia {M}athematica. Preprint available at
  \url{https://arxiv.org/abs/1908.01448}, 2019.

\bibitem{Fefferman73b}
C.~Fefferman.
\newblock A note on spherical summation multipliers.
\newblock {\em Israel J. Math.}, 15:44--52, 1973.

\bibitem{GuWaZh20}
L.~Guth, H.~Wang, and R.~Zhang.
\newblock A sharp square function estimate for the cone in {$\Bbb {R}^3$}.
\newblock {\em Ann. of Math. (2)}, 192(2):551--581, 2020.

\bibitem{HaPoRo20}
A.~Hassell, P.~Portal, and J.~Rozendaal.
\newblock Off-singularity bounds and {H}ardy spaces for {F}ourier integral
  operators.
\newblock {\em Trans. Amer. Math. Soc.}, 373(8):5773--5832, 2020.

\bibitem{Hassell-Rozendaal20}
A.~Hassell and J.~Rozendaal.
\newblock ${L}^{p}$ and $\mathcal{H}^{p}_{FIO}$ regularity for wave equations
  with rough coefficients.
\newblock Preprint available at \url{https://arxiv.org/abs/2010.13761}, 2020.

\bibitem{HeNaSe11}
Y.~Heo, F.~Nazarov, and A.~Seeger.
\newblock Radial {F}ourier multipliers in high dimensions.
\newblock {\em Acta Math.}, 206(1):55--92, 2011.

\bibitem{Hormander09}
L.~H\"{o}rmander.
\newblock {\em The analysis of linear partial differential operators. {IV}}.
\newblock Classics in Mathematics. Springer-Verlag, Berlin, 2009.
\newblock Fourier integral operators, Reprint of the 1994 edition.

\bibitem{Krantz-Parks13}
S.~Krantz and H.~Parks.
\newblock {\em The implicit function theorem}.
\newblock Modern Birkh\"auser Classics. Birkh\"auser/Springer, New York, 2013.
\newblock History, theory, and applications, Reprint of the 2003 edition.

\bibitem{Miyachi80a}
A.~Miyachi.
\newblock On some estimates for the wave equation in {$L^{p}$}\ and {$H^{p}$}.
\newblock {\em J. Fac. Sci. Univ. Tokyo Sect. IA Math.}, 27(2):331--354, 1980.

\bibitem{Peral80}
J.~C. Peral.
\newblock {$L^{p}$}\ estimates for the wave equation.
\newblock {\em J. Funct. Anal.}, 36(1):114--145, 1980.

\bibitem{Rozendaal20}
J.~Rozendaal.
\newblock Rough pseudodifferential operators on {H}ardy spaces for {F}ourier
  integral operators.
\newblock To appear in {J}ournal d'{A}nalyse {M}ath\'{e}matique. Preprint
  available at \url{https://arxiv.org/abs/2010.13895}, 2020.

\bibitem{Rozendaal21}
J.~Rozendaal.
\newblock Characterizations of {H}ardy spaces for {F}ourier integral operators.
\newblock {\em Rev. Mat. Iberoam.}, 37(5):1717--1745, 2021.

\bibitem{Rozendaal22}
J.~Rozendaal.
\newblock Rough pseudodifferential operators on {H}ardy spaces for {F}ourier
  integral operators {II}.
\newblock {\em J. Fourier Anal. Appl.}, 28(4):Paper No. 65, 27, 2022.

\bibitem{Rozendaal-Schippa22}
J.~Rozendaal and R.~Schippa.
\newblock Nonlinear wave equations with slowly decaying initial data.
\newblock Preprint available at \url{https://arxiv.org/abs/2203.06412}, 2022.

\bibitem{Schippa22}
R.~Schippa.
\newblock On smoothing estimates in modulation spaces and the nonlinear
  {S}chr\"{o}dinger equation with slowly decaying initial data.
\newblock {\em J. Funct. Anal.}, 282(5):Paper No. 109352, 46, 2022.

\bibitem{SeSoSt91}
A.~Seeger, C.~D. Sogge, and E.~M. Stein.
\newblock Regularity properties of {F}ourier integral operators.
\newblock {\em Ann. of Math. (2)}, 134(2):231--251, 1991.

\bibitem{Smith98a}
H.~Smith.
\newblock A {H}ardy space for {F}ourier integral operators.
\newblock {\em J. Geom. Anal.}, 8(4):629--653, 1998.

\bibitem{Sogge91}
C.~Sogge.
\newblock Propagation of singularities and maximal functions in the plane.
\newblock {\em Invent. Math.}, 104(2):349--376, 1991.

\bibitem{Sogge17}
C.~Sogge.
\newblock {\em Fourier integrals in classical analysis}, volume 210 of {\em
  Cambridge Tracts in Mathematics}.
\newblock Cambridge University Press, Cambridge, second edition, 2017.

\bibitem{Stein93}
E.~M. Stein.
\newblock {\em Harmonic analysis: real-variable methods, orthogonality, and
  oscillatory integrals}, volume~43 of {\em Princeton Mathematical Series}.
\newblock Princeton University Press, Princeton, NJ, 1993.
\newblock With the assistance of Timothy S. Murphy, Monographs in Harmonic
  Analysis, III.

\bibitem{Wolff00}
T.~Wolff.
\newblock Local smoothing type estimates on {$L^p$} for large {$p$}.
\newblock {\em Geom. Funct. Anal.}, 10(5):1237--1288, 2000.

\end{thebibliography}

\end{document}